\documentclass{amsart}
\usepackage[utf8]{inputenc}
\usepackage{todonotes} 
\usepackage{amsmath, amssymb, amsthm} 
\usepackage{latexsym} 
\usepackage{dsfont}
\usepackage[flushleft]{paralist}
\usepackage{tikz-cd}
\usepackage{hyperref} 

\title{Fine Selmer groups of congruent $p$-adic Galois representations}
\author{Sören Kleine} 
\address{Institut für Theoretische Informatik, Mathematik und Operations Research, Universität der Bundeswehr München, Werner-Heisenberg-Weg 39, 85577 Neubiberg, Germany} 
\email{soeren.kleine@unibw.de} 
\author{Katharina Müller} 
\address{Mathematisches Institut, Georg-August-Universität Göttingen, Bunsenstraße 3-5, 
37073 Göttingen, Germany} 
\email{katharina.mueller@mathematik.uni-goettingen.de} 


 	\newcommand{\Z}{\mathds{Z}}
 	\newcommand{\N}{\mathds{N}}
 	\newcommand{\Q}{\mathds{Q}}

 	\newcommand{\F}{\mathds{F}}
 	 
 	\newcommand{\q}{\mathfrak{q}}

 	\newcommand{\Gal}{\textup{Gal}}
 	\newcommand{\rg}{\textup{rank}}

 	\newcommand{\Ok}{\mathcal{O}}
 	 
 	\newcommand{\coker}{\mathrm{coker}}

 	\newcommand{\Sel}{\textup{Sel}} 

 	\newtheorem{lemma}{Lemma}[section]

 	\newtheorem{thm}[lemma]{Theorem} 
 	\newtheorem*{thm*}{Theorem} 
 	\newtheorem{cor}[lemma]{Corollary}

 	\theoremstyle{definition}

 	\newtheorem{rem}[lemma]{Remark} 
 	 
 	\newtheorem{example}[lemma]{Example} 
 	 
\begin{document}

\maketitle

\begin{abstract} 
   We compare the Pontryagin duals of fine Selmer groups of two congruent $p$-adic Galois representations  over admissible pro-$p$, $p$-adic Lie extensions $K_\infty$ of number fields $K$. We prove that in several natural settings the $\pi$-primary submodules of the Pontryagin duals are pseudo-isomorphic over the Iwasawa algebra; if the coranks of the fine Selmer groups are not equal, then we can still prove inequalities between the $\mu$-invariants. In the special case of a $\Z_p$-extension $K_\infty/K$, we also compare the Iwasawa $\lambda$-invariants of the fine Selmer groups, even in situations where the $\mu$-invariants are non-zero. Finally, we prove similar results for certain abelian non-$p$-extensions. 
\end{abstract} 

\section{Introduction} 
Let $p$ be a prime, and let $F$ be a finite extension of $\Q_p$, with ring of integers $\Ok$ and uniformising element $\pi$. Suppose that $V_1$ and $V_2$ denote two $F$-representations of the absolute Galois group of a fixed number field $K$, and that $T_1 \subseteq V_1$ and $T_2 \subseteq V_2$ are two Galois invariant sublattices. We let $A_1 = V_1/T_1$ and $A_2 = V_2/T_2$ and we assume that $A_1[\pi^l]$ and $A_2[\pi^l]$ 
are isomorphic as Galois modules for some $l \in \N$. In this article, we study the Pontryagin duals of the fine Selmer groups of $A_1$ and $A_2$ over (strongly) admissible $p$-adic Lie extensions, and we compare their ranks and Iwasawa invariants. 

By an \emph{admissible $p$-adic Lie extension} we mean a normal extension $K_\infty/K$ such that only finitely many primes of $K$ ramify in $K_\infty$ and such that ${G = \Gal(K_\infty/K)}$ is a compact, pro-$p$, $p$-adic Lie group without $p$-torsion. For any finite set $\Sigma$ of primes of $K$, an admissible $p$-adic Lie extension $K_\infty/K$ shall be called \emph{strongly $\Sigma$-admissible} if $K_\infty$ contains a $\Z_p$-extension $L$ of $K$ such that no prime $v \in \Sigma$ and no prime of $K$ which ramifies in $K_\infty$ is completely split in $L$ (see also Section~\ref{section:notation}; in the literature usually only the case of the cyclotomic $\Z_p$-extension $L = K_\infty^c$ of $K$ is considered, see for example \cite{lim}). 

The comparison of \emph{Selmer groups} of congruent $p$-adic representations goes back to the seminal work of Greenberg and Vatsal (see \cite{greenberg-vatsal}), who considered elliptic curves defined over $\Q$ with good and ordinary reduction at some odd prime $p$ (in fact the Selmer groups were studied more generally in the context of Galois representations). The main issue dealt with in the article \cite{greenberg-vatsal} is the relation between algebraically and analytically (i.e.~via $p$-adic $L$-functions) defined Iwasawa invariants. Roughly speaking, Greenberg and Vatsal treated the $\mu = 0$ case and only considered the cyclotomic $\Z_p$-extension. 

Over the last years, the results in \cite{greenberg-vatsal} have been generalised in many different ways and we only mention a few exemplary results. For the comparison of analytical invariants of congruent elliptic curves defined over $\Q$, we refer to \cite{hatley}; in the present article we stick to the algebraic side. Most authors have focussed on the $\mu = 0$ setting from \cite{greenberg-vatsal}: if $\mu = 0$ for the Selmer group of $A_1$, then the same holds true for the Selmer group of $A_2$. Moreover, over $\Z_p$-extensions one can then often prove equality of $\lambda$-invariants (we refer to Section~\ref{section:notation} for the definition of the Iwasawa invariants). Analogous results have been obtained for Selmer groups of Galois representations over the anticyclotomic $\Z_p$-extension of an imaginary quadratic base field $K$ (see \cite{hatley-lei}) and for \emph{signed Selmer groups} of Galois representations over the cyclotomic $\Z_p$-extension of a number field in the non-ordinary setting (see for example \cite[Section~3]{ponsinet}). Moreover, there exist vast generalisations to Selmer groups attached to families of modular forms (see for example \cite{emerton-pollack-weston}, \cite{sharma} and \cite{barth}). 

Situations where $\mu \ne 0$ have been studied for example in \cite{ahmed-shekhar} and \cite{barman-saikia}. In these articles the authors considered congruent elliptic curves $E_1$ and $E_2$ over $\Q$ at primes $p > 2$ of good ordinary reduction. Under the additional assumption that $E_j(\Q)[p^\infty] = \{0\}$ for $j \in \{ 1,2 \}$, the authors deduced the equality of $\lambda$-invariants (see \cite{ahmed-shekhar}), respectively $\mu$-invariants (see \cite{barman-saikia}) from a sufficiently high congruence relation $E_1[p^l] \cong E_2[p^l]$. Much more generally, Lim studied the Selmer groups of Galois representations over admissible $p$-adic Lie extensions in \cite{lim}. In particular, he obtained the following result: if $A_1$ and $A_2$ are attached to two $p$-adic Galois representations and $A_1[\pi^l] \cong A_2[\pi^l]$ for some sufficiently large $l$, then the $\pi$-primary submodules of the Pontryagin duals of the associated Selmer groups are pseudo-isomorphic. This comparison statement is much stronger than the previous results. We are able to prove a similar result for fine Selmer groups (see Theorem~\ref{thm:A} below). Lim also studied \emph{strict Selmer groups}, as introduced by Greenberg in \cite{greenberg89}. These strict Selmer groups of $p$-adic Galois representations have also been studied by Hachimori in \cite{hachimori}. 

In the present article, our main objective is the comparison of \emph{fine Selmer groups} of congruent $p$-adic Galois representations over admissible $p$-adic Lie extensions. These objects have previously been investigated by Lim and Sujatha in \cite{lim_congruent}, who obtained a comparison result in the $\mu = 0$ setting  under a stronger condition on the decomposition of primes in $K_\infty/K$ (see \cite[Theorem~3.7]{lim_congruent}). Moreover, Jha studied in \cite{jha} the invariance of several arithmetic properties of fine Selmer groups of modular forms in a branch of a Hida family in the $\mu = 0$ setting. 

Our first main result is an analogon of the strong result of Lim in \cite{lim} for fine Selmer groups over strongly admissible $p$-adic Lie extensions, which is not restricted to the case $\mu = 0$. For an admissible $p$-adic Lie extension $K_\infty$ of $K$ and an $F$-representation $V$ of the absolute Galois group of $K$, we let $T$ denote a Galois invariant $\Ok$-lattice in $V$ and set $A =V/T$. Let $\Sigma$ be a finite set of primes of $K$ containing all the primes above $p$ and each prime where $V$ is ramified. Then $Y_{A, \Sigma}^{(K_\infty)}$ shall denote the Pontryagin dual of the  $\Sigma$-fine Selmer group of $A$ over $K_\infty$ (see Section~\ref{subsection:selmergroups} for the precise definition). 
\begin{thm} \label{thm:A} 
   Let $A_1$ and $A_2$ be associated with two $F$-representations $V_1$ and $V_2$ of the absolute Galois group of the number field $K$. If $p = 2$, then we assume that $K$ is totally imaginary. Let $\Sigma$ be a finite set of primes of $K$ which contains the primes above $p$ and the sets of primes of $K$ where either $V_1$ or $V_2$ is ramified. Let $K_\infty/K$ be a strongly $\Sigma$-admissible $p$-adic Lie extension, and let $G = \Gal(K_\infty/K)$. 
   
   We let $r_j = \rg_{\Ok[[G]]}(Y_{A_j, \Sigma}^{(K_\infty)})$, $1 \le j \le 2$. Let $l$ be the minimal integer such that $(\pi^l Y_{A_1, \Sigma}^{(K_\infty)})[\pi]$ is pseudo-null over $\Ok[[G]]$. 
   Then the following statements hold. \begin{compactenum}[(a)]  
     \item If $A_1[\pi^l]\cong A_2[\pi^l]$ as $G_{K}$-modules and $r_2 \le r_1$, then $\mu(Y_{A_1, \Sigma}^{(K_\infty)})\le \mu(Y_{A_2, \Sigma}^{(K_\infty)})$.  
     \item If $A_1[\pi^{l+1}]\cong A_2[\pi^{l+1}]$, then $r_2 \le r_1$. 
     If moreover $r_2 = r_1$, then $$\mu(Y_{A_1, \Sigma}^{(K_\infty)}) = \mu(Y_{A_2, \Sigma}^{(K_\infty)})$$ and the modules $Y_{A_1, \Sigma}^{(K_\infty)}[\pi^{\infty}]$ and $Y_{A_2, \Sigma}^{(K_\infty)}[\pi^{\infty}]$ are pseudo-isomorphic. 
     \item In particular, if $A_1[\pi]\cong A_2[\pi]$ and $r_2 = r_1$, then ${\mu(Y_{A_1, \Sigma}^{(K_\infty)})=0}$ holds if and only if $\mu(Y_{A_2, \Sigma}^{(K_\infty)})=0$. 
     \item If $A_1[\pi^{l+1}] \cong A_2[\pi^{l+1}]$ for some integer $l$ such that both $(\pi^l Y_{A_j, \Sigma}^{(K_\infty)})[\pi]$, ${1 \le j \le 2}$, are pseudo-null, then $r_2 = r_1$ and ${\mu(Y_{A_1, \Sigma}^{(K_\infty)}) = \mu(Y_{A_2, \Sigma}^{(K_\infty)})}$. 
     \end{compactenum} 
\end{thm} 

This result will be proved in Section~\ref{sec:generic}. The basic idea of the proof is to relate the $\pi^k$-torsion subgroups of the fine Selmer groups of $A_j$, $k \in \N$, to certain $\pi^k$-fine Selmer groups (defined in Section~\ref{section:notation}) which depend only on $A_j[\pi^k]$. In the case of admissible $p$-adic Lie extensions $K_\infty/K$ which are not strongly admissible, we can derive similar results under the hypothesis that $A_j(K_v)[\pi]=\{0\}$ for every $v \in \Sigma$ and $j \in \{1,2\}$ (see Theorem~\ref{thm:A2} below). In order to obtain this result, we use an argument which goes back to the paper of Greenberg and Vatsal (see \cite[Proposition~2.8]{greenberg-vatsal}), appears also in work of Mazur and Rubin (see \cite[Lemma~3.5.3]{kolyvagin}) and has been used in, e.g.  \cite{barman-saikia} and \cite{ponsinet}. This approach is of particular interest if one wants to treat $\Z_p$-extensions $K_\infty$ of $K$ in which some prime above $p$ or a ramified prime is completely split. In the special case of $\Z_p$-extensions, and under the additional hypotheses on the $\pi$-torsion which have been mentioned above, we can in fact go one step further and obtain results on the $\lambda$-invariants, provided that the $\Ok[[G]]$-modules $Y_{A_i, \Sigma}^{(K_\infty)}$ both are torsion: 
\begin{thm}  \label{thm:B} 
  In the setting of Theorem~\ref{thm:A2}, suppose that $G \cong \Z_p$ and that both ranks $r_1$ and $r_2$ are zero. Then, in addition to the assertions from Theorem~\ref{thm:A2}, the following two statements hold: \begin{compactenum}[(a)] 
     \item If $l \in \N$ is large enough such that $(\pi^l Y_{A_1, \Sigma}^{(K_\infty)})[\pi] = \{0\}$ and ${A_1[\pi^{l+1}] \cong A_2[\pi^{l+1}]}$, then ${\lambda(Y_{A_2, \Sigma}^{(K_\infty)}) \le \lambda(Y_{A_1, \Sigma}^{(K_\infty)})}$. 
     \item If $A_1[\pi^{l+1}] \cong A_2[\pi^{l+1}]$ for some $l$ such that both $(\pi^l Y_{A_j, \Sigma}^{(K_\infty)})[\pi] = \{0\}$, ${1 \le j \le 2}$, then ${\lambda(Y_{A_1, \Sigma}^{(K_\infty)}) = \lambda(Y_{A_2, \Sigma}^{(K_\infty)})}$. 
  \end{compactenum} 
\end{thm} 

We remark that we do not have to assume that the $\mu$-invariants vanish in Theorem~\ref{thm:B}. 

Finally, in Section~\ref{section:q}, we consider certain abelian non-$p$-extensions $K_\infty$ of $K$. In two different settings (inspired by the two different cases treated in Section~\ref{section:fine}), we compare the $\Ok$-ranks of $Y_{A_1, \Sigma}^{(K_\infty)}$ and $Y_{A_2, \Sigma}^{(K_\infty)}$ and derive (in-)equalities analogous to those in Theorem~\ref{thm:B}. We also remark that the group ring $\Ok[[\Gal(K_\infty/K)]]$ is not well-behaved in this situation and the $\Ok$-rank is the natural substitute for the notion of Iwasawa $\lambda$-invariants. 

\textbf{Acknowledgements.} We thank Dominik Bullach and Antonio Lei for their valuable comments on an earlier draft of this article.  

\section{Background and notation} \label{section:notation} 
\subsection{Admissible $p$-adic Lie extensions and Iwasawa modules} \label{subsection:modules} 
We fix once and for all a rational prime $p$. Let $F$ be a finite extension of $\Q_p$. We will denote its ring of integers by $\mathcal{O}$ and a generator of its maximal ideal by $\pi$. Note that $\mathcal{O}/(\pi)$ is a finite field with $q = p^f$ elements, where $f$ is the inertia degree of $p$ in $F/\Q_p$. 
For any Noetherian $\mathcal{O}$-module $G$, we denote by $G[\pi^\infty]$ the subgroup of $\pi$-power torsion elements; for any $i \in \N$, $G[\pi^i]$ shall denote the subgroup of elements which are annihilated by $\pi^i$. 

In this article, an \emph{admissible $p$-adic Lie extension} $K_\infty$ of a number field $K$ will always be a normal extension $K_\infty/K$ such that 
\begin{compactitem}
  \item $G := \Gal(K_\infty/K)$ is a compact pro-$p$ $p$-adic Lie group, 
  \item $G[p^\infty] = \{0\}$, i.e. $G$ does not contain any $p$-torsion elements, and 
  \item the set $S_{\textup{ram}}(K_\infty/K)$ of primes of $K$ ramifying in $K_\infty$ is finite. 
\end{compactitem} 
Let $\Sigma$ be a finite set of finite primes of $K$. The pro-$p$-extension $K_\infty/K$ is called \emph{strongly $\Sigma$-admissible} if it is admissible and moreover contains a  $\Z_p$-extension $L$ of $K$ such that no prime in $\Sigma\cup S_{\textup{ram}}(K_{\infty}/K)$ is completely split in $L$. In this case, we fix $L$ and denote by $H \subseteq G$ the subgroup fixing $L$. Note that any strongly $\Sigma$-admissible $p$-adic Lie extension $K_\infty/K$ is strongly $\Sigma\cup S_{\textup{ram}}(K_{\infty}/K)$-admissible. By abuse of notation we will always assume that $\Sigma$ contains $S_{\textup{ram}}(K_{\infty}/K)$ if $K_{\infty}/K$ is a strongly $\Sigma$-admissible $p$-adic Lie extension. 

An admissible $p$-adic Lie extension $K_\infty/K$ is called \emph{strongly admissible} if it contains the cyclotomic $\Z_p$-extension of $K$. Since no prime of $K$ splits completely in the cyclotomic $\Z_p$-extension, a strongly admissible $p$-adic Lie extension is strongly $\Sigma$-admissible for every finite set $\Sigma$. 

If $K_\infty/K$ is an admissible $p$-adic Lie extension, then the completed group ring ${\mathcal{O}[[G]] = \mathcal{O} \otimes_{\Z_p} \Z_p[[G]]}$ is a Noetherian domain (see \cite[Theorem~2.3]{coates-howson}), and we can define the $\mathcal{O}[[G]]$-rank of a finitely generated $\mathcal{O}[[G]]$-module $X$ by 
\[ \rg_{\mathcal{O}[[G]]}(X) = \dim_{\mathcal{F}(G)}(\mathcal{F}(G) \otimes_{\mathcal{O}[[G]]} X), \] 
where $\mathcal{F}(G)$ denotes the skew field of fractions of $\mathcal{O}[[G]]$ (see \cite[Chapter~10]{goodearl-warfield}). 

Following Howson (see \cite[(33)]{howson}),  we define the $\mu$-invariant of a finitely generated $\mathcal{O}[[G]]$-module $X$ as 
\begin{align} \label{eq:mu} 
\mu(X) = \sum_{i \ge 0} \rg_{\F_q[[G]]}(\pi^i X[\pi^\infty]/\pi^{i+1} X[\pi^\infty]);
\end{align} 
this is a finite sum as $X$ is Noetherian. 
\begin{rem}
  \label{finitesum}
  Let $X$ be a Noetherian $\pi$-primary $\mathcal{O}[[G]]$-module. Then there exists an integer $m$ such that $\pi^mX=\{0\}$. Suppose now that ${\rg_{\F_q[[G]]}( X[\pi])=0}$. Then there exists an annihilator ${f\in \mathcal{O}[[G]]\setminus \pi\mathcal{O}[[G]]}$ of $X[\pi]$. In particular, $$\pi^{m-1}fX=\{0\} \quad \text{ and } \quad \pi^{m-2}fX\subseteq X[\pi]. $$ 
  Thus, we inductively obtain that ${f^mX=\{0\}}$. Therefore  ${\rg_{\F_q[[G]]}(X/\pi X)=0}$.
\end{rem}
\begin{lemma} \label{lemma:mu} 
  Let $G = \Gal(K_\infty/K)$ be as above, and let $X$ be a finitely generated $\mathcal{O}[[G]]$-module of rank $r$. Then 
  $$ \rg_{\F_q[[G]]}(\pi^i X[\pi^\infty]/\pi^{i+1}X[\pi^\infty]) = \rg_{\F_q[[G]]}(\pi^iX/\pi^{i+1}X) - r$$ 
  for each $i \in \N$. 
\end{lemma} 
\begin{proof} 
  By \cite[Proposition~4.12]{lim_fuer_howson} 
  $$ \rg_{\F_q[[G]]}(\pi^iX/\pi^{i+1}X) = \rg_{\F_q[[G]]}((\pi^iX)[\pi]) + \rg_{\mathcal{O}[[G]]}(\pi^i X)$$ 
  and 
  $$ \rg_{\F_q[[G]]}(\pi^iX[\pi^\infty]/\pi^{i+1}X[\pi^\infty]) = \rg_{\F_q[[G]]}((\pi^iX[\pi^\infty])[\pi]) + \rg_{\mathcal{O}[[G]]}(\pi^i X[\pi^\infty]). $$ 
  Now $(\pi^i X)[\pi] = (\pi^i X[\pi^\infty])[\pi]$, $\rg_{\mathcal{O}[[G]]}(\pi^iX) = r$ and 
  $\rg_{\mathcal{O}[[G]]}(\pi^iX[\pi^\infty]) = 0$.  
\end{proof} 

The most important class of admissible $p$-adic Lie extensions are the $\Z_p$-ex\-ten\-sions. A \emph{$\Z_p$-extension} $K_\infty/K$ is a normal extension such that $G = \Gal(K_\infty/K)$ is isomorphic to the additive group of $p$-adic integers. In this special case, the theory of finitely generated $\mathcal{O}[[G]]$-modules is very well understood: the completed group ring $\mathcal{O}[[G]]$ is isomorphic to the ring $\Lambda:=\mathcal{O}[[T]]$ of formal power series in one variable. Each finitely generated torsion $\Lambda$-module $X$ is pseudo-isomorphic to an \emph{elementary} $\Lambda$-module of the form 
$$ E_X = \bigoplus_{i=1}^s \Lambda/(\pi^{e_i}) \oplus \bigoplus_{j = 1}^t \Lambda/(h_j), $$ 
where $h_1, \ldots, h_t \in \Lambda$ are so-called distinguished polynomials. Here \emph{pseudo-\-iso\-mor\-phic} means that there exists a $\Lambda$-module homomorphism $\varphi: X \longrightarrow E_X$ with finite kernel and cokernel. One defines the \emph{(classical) Iwasawa invariants} of $X$ by ${\mu(X) := \sum_{i=1}^s e_i}$ and ${\lambda(X) := \sum_{j=1}^t \deg(h_j)}$. This notation is well-defined since the classical $\mu$-invariant coincides with the invariant $\mu(X)$ given in \eqref{eq:mu} above in the special case of $\Z_p$-extensions:  
\begin{lemma}  
\label{lemma:invariants} 
  Let $X$ be a finitely generated $\Lambda$-module. Then the classical Iwasawa $\mu$-invariant is equal to 
  \[\sum_{i=0}^{\infty} \rg_{\F_q[[T]]}(\pi^i X[\pi^\infty] /\pi^{i+1}X[\pi^\infty]). \]\end{lemma} 
  \begin{proof} 
  This proof is well-known (see, e.g.~\cite[Section~3.4]{venjakob}), but we recall it for the convenience of the reader. Let $E$ be an elementary $\Lambda$-module that is pseudo-isomorphic to $X$. Then we can write $E=\Lambda^r \oplus \bigoplus_{i=1}^s\Lambda/(\pi^{e_i})\oplus E_{\lambda}$ for a torsion $\Lambda$-module $E_{\lambda}$ which is a finitely generated free $\mathcal{O}$-module. Therefore the classical Iwasawa invariants can be computed as 
  \[ \mu(X)=\mu(E)= \sum_{i=0}^{\infty}\vert \{k\mid e_k \ge i+1\}\vert=\sum_{i=0}^{\infty} \rg_{\F_q[[T]]}(\pi^{i}X[\pi^\infty]/\pi^{i+1}X[\pi^\infty])\]
because 
\begin{eqnarray*} \rg_{\F_q[[T]]}(\pi^iX[\pi^\infty]/\pi^{i+1}X[\pi^\infty]) & = &  \rg_{\F_q[[T]]}(\pi^iE[\pi^\infty]/\pi^{i+1}E[\pi^\infty]) \\ 
& = & \vert \{k \mid e_k \ge i+1\} \vert\end{eqnarray*}  
for every $i \in \N$. 
\end{proof} 

\subsection{Fine Selmer groups} \label{subsection:selmergroups} 
For any discrete $\Z_p$-module $M$, we define the \emph{Pontryagin dual} of $M$ as 
\[ M^\vee = \textup{Hom}_{\textup{cont}}(M, \Q_p/\Z_p) \] 
(i.e. the set of continuous homomorphisms). 

If $K$ is a number field and $v$ denotes any prime of $K$, then $K_v$ will always denote the completion of $K$ at $v$. We denote by $G_K$ the Galois group $\Gal(\overline{K}/K)$, where $\overline{K}$ denotes a fixed algebraic closure of $K$. If $M$ is any $G_K$-module, then we let $H^i(K,M) := H^i(G_K,M)$ denote the corresponding Galois cohomology groups, $i \in \N$. Moreover, if $L/K$ is an algebraic extension, then we write $H^i(L/K,M) = H^i(\Gal(L/K),M)$ for brevity. 

Now fix a number field $K$. Let $V$ be a finite dimensional $F$-representation of $\Gal(\overline{K}/K)$ for some fixed algebraic closure $\overline{K}$ of $K$. Let $T$ be a Galois stable $\mathcal{O}$-lattice in $V$ and write $A=V/T$. Note that $A$ is as $\Ok$-module isomorphic to $(F/\mathcal{O})^l$ for some non-negative integer $l=\textup{dim}(V)$. In particular, $A = A[\pi^\infty]$ is $\pi$-primary, i.e. each element of the $\Ok$-module $A$ is annihilated by some power of $\pi$. By abuse of notation we will also refer to $l$ as the dimension of $A$. 

We denote by $S_p$ and $S_{\textup{ram}}(A)$ the set of primes of $K$ over $p$ and the set of primes of $K$ where $V$ is ramified. For any number field $L \supseteq K$ that is unramified outside $S_p\cup S_{\textup{ram}}(A)$ we denote by $A(L)$ the maximal submodule of $A$ on which $\Gal(\overline{K}/L)$ acts trivially. 

We mention an important and classical special case: let $A$ be an abelian variety defined over the number field $K$. We assume that $F = \Q_p$, i.e. $\Ok = \Z_p$. Let $T = T_p(A) = \varprojlim_n A[p^n]$ be the Tate module of $A$ and $V = T_p(A) \otimes_{\Z_p} \Q_p$; then $V/T \cong A[p^\infty]$. In this setting, for any field $L$, the group $A(L)[p^\infty]$ is the usual group of $L$-rational $p$-torsion points on $A$. Moreover, the ramified primes correspond to the primes of $K$ where $A$ has bad reduction, by the criterion of Néron-Ogg-Shafarevich (see \cite[Theorem~IV.4.1]{lang_diophantine}). 

For the number field $K$, $A = V/T$ as above and a prime number $p$, we define the ($\pi$-primary part of the) \emph{fine Selmer group} of $A$ over $K$ as 
\[ \textup{Sel}_{0,A}(K) = \ker \left( H^1(K, A) \longrightarrow \prod_v H^1(K_{v}, A) \right).\] 
In our applications it will be more convenient to work with the following definition: 
\[ \textup{Sel}_{0,A,\Sigma}(K) = \ker \left( H^1(K_\Sigma/K, A) \longrightarrow \prod_{v \in \Sigma} H^1(K_v, A) \right) \]  
for suitable (usually finite) sets $\Sigma$ of primes of $K$ containing all the ramified primes of the representation $V$ and all primes above $p$. Let $K_{\Sigma}$ be the maximal algebraic extension of $K$ unramified outside the primes in $\Sigma$. If $L\subseteq K_{\Sigma}$ is any, non-necessarily finite,  extension, then we define 
\[\textup{Sel}_{0,A,\Sigma}(L)=\varinjlim_{K\subseteq L'\subseteq L}\textup{Sel}_{0,A,\Sigma}(L'), \]
where $L'$ runs through all finite subextensions $K\subseteq L'\subseteq L$. Here we note that $K_\Sigma = L'_\Sigma$, since $L/K$ is unramified outside of $\Sigma$, and therefore each $\textup{Sel}_{0,A,\Sigma}(L')$ is a subgroup of $H^1(K_\Sigma/K, A)$. 

A priori this definition depends on the choice of $\Sigma$. But if the cyclotomic $\Z_p$-extension of $K$, denoted by $K_{\infty}^c$, is contained in $L$, then the definition becomes independent of the set $\Sigma$ by a result of Sujatha and Witte (see \cite[Section~3]{ramdorai-witte}). In fact their proof depends only on the fact that none of the primes in $\Sigma$ is totally split in $K^c_{\infty}/K$. Therefore the definition of the fine Selmer group does not depend on the choice of $\Sigma$ if we consider strongly $\Sigma$-admissible extensions $K_{\infty}/K$.

Finally, we define  $\pi^i$-fine Selmer groups, $i \in \N$, as 
\[ \textup{Sel}_{0,A[\pi^i],\Sigma}(K) = \ker \left( H^1(K_\Sigma/K, A[\pi^i]) \longrightarrow \prod_{v \in \Sigma} H^1(K_{v}, A[\pi^i]) \right),\] 
where $\Sigma$ is as above. Note: these $\pi^i$-fine Selmer groups may depend on the choice of $\Sigma$ even for algebraic extensions $L$ of $K$ which contain the cyclotomic $\Z_p$-extension $K_\infty^c$ (see \cite[proof of Theorem~5.1]{lim-murty} for an example for abelian varieties). 

Now let $K_\infty/K$ be an admissible $p$-adic Lie extension, and let $\Sigma$ be a finite set of primes of $K$ which contains ${S_{\textup{ram}}(K_\infty/K) \cup S_p \cup S_{\textup{ram}}(A)}$ (if $p = 2$, then we assume that $K$ is totally imaginary).  Then we can define fine Selmer groups of $A$ over each number field $L \subseteq K_\infty$ containing $K$. We denote the corresponding Pontryagin duals by 
\[ Y_{A,\Sigma}^{(L)} = \textup{Sel}_{0,A,\Sigma}(L)^\vee, \] 
and we define the projective limit
\[  Y_{A,\Sigma}^{(K_\infty)} = \varprojlim_{K \subseteq L \subseteq K_\infty} Y_{A,\Sigma}^{(L)}\] 
with respect to the corestriction maps (where $L$ runs over the finite subextensions of $K_\infty/K$).

\section{Fine Selmer groups of congruent representations} \label{section:fine} 
The aim of this section is to study the relation between the Iwasawa invariants of the fine Selmer groups associated with two representations $V_1$ and $V_2$ defined over the same number field $K$. The representations we consider will always satisfy a congruence condition, meaning that $A_1[\pi^l]$ and $A_2[\pi^l]$ are isomorphic as $G_K$-modules for some integer $l$ (where $A_i = V_i/T_i$ as usual). Note that this implies that the two representations have the same dimension $d$. We will always fix a set $\Sigma$ of primes in $K$ containing all ramified places for $A_1$ and $A_2$, and all places above $p$. Let $K_{\infty}/K$ be an admissible $p$-adic Lie extension. We consider two cases: 
\begin{compactenum}[i)]
\item $K_{\infty}/K$ is strongly $\Sigma$-admissible (Section \ref{sec:generic}).
\item $K_{\infty}/K$ is admissible and $A(K_v)[\pi]=0$ for all $v\in \Sigma$ (Section \ref{sec:split}).
\end{compactenum}
Note that case ii) only becomes relevant if a prime of $\Sigma$ is completely split in $K_{\infty}/K$.

\subsection{The generic case}
\label{sec:generic}
In this section we prove Theorem \ref{thm:A}. The main ingredient in the proof is a relation between $\textup{Sel}_{0,A}(L)[\pi^l]$ and $\textup{Sel}_{0,A[\pi^l]}(L)$ for any finite subextension $K\subseteq L\subseteq K_{\infty}$ of the $p$-adic Lie extension $K_{\infty}/K$.

\begin{lemma} \label{lemma:vergleich} 
Let $A$ be associated with a representation of $G_K$ of dimension $d$ and let $\Sigma$ be a finite set of primes of $K$ containing $S_p$ and $S_{\textup{ram}}(A)$. If $p = 2$, then we assume that $K$ is totally imaginary. Let $L/K$ be a finite extension which is contained in $K_\Sigma$. Then 
\[\vert v_p(|\textup{Sel}_{0,A,\Sigma}(L))[\pi^k]|)-v_p(|\textup{Sel}_{0,A[\pi^k],\Sigma}(L)|)\vert \le fdk(1+|\Sigma(L)|) \] 
for each integer $k \ge 1$, where $\Sigma(L)$ denotes the set of primes of $L$ above $\Sigma$ and $f$ is the inertia degree of $p$ in $F/\Q_p$. 
\end{lemma}
\begin{rem}
  Note that if $A$ is the $p$-primary part of an abelian variety of dimension $d$ then the corresponding representation has dimension $2d$. 
\end{rem}
\begin{proof}
We start from the following commutative diagram
\[\begin{tikzcd}
0\arrow[r]&\textup{Sel}_{0,A[\pi^k]}(L)\arrow[r]\arrow[d,"s"]&H^1(K_{\Sigma}/L,A[\pi^k])\arrow[r]\arrow[d,"h"]& \bigoplus_{v\in \Sigma(L)}H^1({L_{v}},A[\pi^k])\arrow[d,"g"] \\0 \arrow[r]&\textup{Sel}_{0,A}(L)[\pi^k]\arrow[r]&H^1(K_{\Sigma}/L,A)[\pi^k]\arrow[r]& \bigoplus_{v\in \Sigma(L)}H^1({L_{v}},A)[\pi^k]\end{tikzcd}
\]
Consider the exact sequence
\[0\longrightarrow A[\pi^k]\longrightarrow A \stackrel{\cdot \pi^k}{\longrightarrow} A\longrightarrow 0.\]
The surjectivity follows from the fact that $A$ is divisible as $\mathcal{O}$-module. Note further that the representation $V$ is unramified outside $\Sigma$. Thus, there is a well-defined action of $\Gal(K_{\Sigma}/L)$ on $A$ and we can take $K_{\Sigma}/L$-cohomology of the exact sequence in order to see that the map $h$ is surjective.
Moreover, 
\begin{eqnarray*} 
   \ker(h) & \cong & \coker(\pi^k\colon H^0(K_{\Sigma}/L,A)\longrightarrow H^0(K_{\Sigma}/L,A)) \\ 
             & = & A(L)/\pi^k A(L). 
\end{eqnarray*} 
The last equality is due to the fact that all ramified primes are contained in $\Sigma$. 
Analogously, we see that $g$ is surjective and that
$$\ker(g)\cong \bigoplus_{v\in \Sigma(L)}A(L_{v})/\pi^k A(L_{v}). $$ 
We obtain the bounds $v_p(|\ker(h)|) \le dkf$ and $v_p(|\ker (g)|)\le dkf|\Sigma(L)|$. Using the exact sequence
\begin{align} \label{eq:star} 0\longrightarrow \ker(s)\longrightarrow \textup{Sel}_{0,A[\pi^k],\Sigma}(L)\longrightarrow \textup{Sel}_{0,A,\Sigma}(L)[\pi^k]\longrightarrow \coker(s)\longrightarrow 0,\end{align} 
we may conclude that $\vert v_p(|\textup{Sel}_{0,A[\pi^k],\Sigma}(L)|)- v_p(|\textup{Sel}_{0,A,\Sigma}(L)[\pi^k]|)\vert$ is bounded by \begin{align*}
 v_p(|\ker(s)|)+v_p(|\coker(s)|)\le v_p(|\ker(h)|)+v_p(|\ker(g)|)\le fdk+dfk|\Sigma(L)|.\end{align*}
\end{proof} 

\begin{cor} \label{cor:vergleich} 
   Let $A$ be associated with a representation of $G_K$, and let $\Sigma$ be a finite set of primes of $K$ containing $S_p$ and $S_{\textup{ram}}(A)$. If $p = 2$, then we assume that $K$ is totally imaginary. Let $K_\infty/K$ be a strongly $\Sigma$-admissible $p$-adic Lie extension. Then 
   $\rg_{\F_q[[G]]}(\pi^i Y_{A,\Sigma}^{(K_\infty)}/\pi^{i+1}Y_{A,\Sigma}^{(K_\infty)})$ equals 
   \[\rg_{\F_q[[G]]}(\varprojlim_{K \subseteq L \subseteq K_\infty} \Sel_{0,A[\pi^{i+1}],\Sigma}(L)^\vee / \varprojlim_{K \subseteq L \subseteq K_\infty} \Sel_{0,A[\pi^{i}],\Sigma}(L)^\vee)\] 
   for every $i \in \N$, where $L$ runs over the finite subfields of $K_\infty/K$. 
\end{cor} 
\begin{proof} 
   Let $k \in \N$. For every finite subextension $L \subseteq K_\infty$ of $K$, we consider the exact sequence 
   \[ 0 \longrightarrow M^{(L)} \longrightarrow \Sel_{0,A,\Sigma}(L)[\pi^k]^\vee \longrightarrow \Sel_{0,A[\pi^k],\Sigma}(L)^\vee \longrightarrow N^{(L)} \longrightarrow 0\] 
   which is obtained from \eqref{eq:star} by taking Pontragin duals. In particular, $N^{(L)}$ is a finite abelian group of order at most $p^{dkf}$, and $M^{(L)} = \bigoplus_{v \in \Sigma(L)} G_v^{(L)}$, where each $G_v^{(L)}$ is a finite abelian group of order at most $p^{dkf}$. 
   
   Taking the projective limits along the $L \subseteq K_\infty$, we obtain an exact sequence 
    \begin{align} \label{eq:vergleich} 0 \longrightarrow M \longrightarrow Y_{A,\Sigma}^{(K_\infty)}/\pi^k Y_{A,\Sigma}^{(K_\infty)} \longrightarrow \varprojlim_{K \subseteq L \subseteq K_\infty} \Sel_{0,A[\pi^k],\Sigma}(L)^\vee \longrightarrow N \longrightarrow 0, \end{align}  
    where $N$ is a finite abelian group and where $M$ is finitely generated over $\mathcal{O}[[H]]$ because no prime $v \in \Sigma$ splits completely in the $\Z_p$-extension $K_\infty^H$ of $K$ which is fixed by $H \subseteq G$. In fact, replacing $K$ by a finite subextension of $K_\infty^H$ if necessary (this does not affect the projective limit), we may assume that actually the primes $v \in \Sigma$ do not split at all in $K_\infty^H/K$. 
    
    Letting $\Gamma := G/H \cong \Z_p$, the group ring $\mathcal{O}[[\Gamma]]$ can be identified with the ring $\Lambda = \mathcal{O}[[T]]$. Since $M$ is finitely generated over $\mathcal{O}[[H]]$, there exists a non-constant annihilator of $M$ in ${\mathcal{O}[[G]] \cong \mathcal{O}[[H]][[T]]}$ by \cite[Proposition~2.3 and Theorem~2.4]{5_people}; in particular, the annihilator is not a power of $\pi$ (note: the result in \cite{5_people} is formulated for the case $\Ok = \Z_p$, but the proof goes through in our more general setting). 
    
    Considering now $k = i$ and $k = i+1$, we may conclude that there exists a non-constant annihilator in $\mathcal{O}[[G]]$ of the  cokernels and kernels of both maps \[Y_{A,\Sigma}^{(K_\infty)}/ \pi^i Y_{A,\Sigma}^{(K_\infty)} \longrightarrow \varprojlim_{K \subseteq L \subseteq K_\infty} \Sel_{0,A[\pi^{i}],\Sigma}(L)^\vee\] and \[Y_{A,\Sigma}^{(K_\infty)} / \pi^{i+1} Y_{A,\Sigma}^{(K_\infty)} \longrightarrow \varprojlim_{K \subseteq L \subseteq K_\infty} \Sel_{0,A[\pi^{i+1}],\Sigma}(L)^\vee. \] 
    Taking quotients proves the assertion of the corollary. 
\end{proof} 

We need one final auxiliary 
\begin{lemma} \label{lemma:G-iso} 
Let $A_1$ and $A_2$ be associated with two representations $V_1$ and $V_2$ of $G_K$, and let $\Sigma$ be a finite set of primes of $K$ which contains ${S_p \cup S_{\textup{ram}}(A_1) \cup S_{\textup{ram}}(A_2)}$. If $p = 2$, then we assume that $K$ is totally imaginary. 
We assume that $A_1[\pi^i]$ and $A_2[\pi^i]$ are isomorphic as $G_K$-modules for some $i \in \N$, $i \ge 1$. 

Then $\textup{Sel}_{0,A_1[\pi^i],\Sigma}(L)\cong \textup{Sel}_{0,A_2[\pi^i],\Sigma}(L)$ for every finite extension $L \subseteq K_\Sigma$ of $K$. 
\end{lemma} 
\begin{proof} 
  Let $\phi \colon A_1[\pi^i] \longrightarrow A_2[\pi^i]$ be a $G_K$-module homomorphism. As $V_1$ and $V_1$ are unramified outside of $\Sigma$, the group $\Gal(\overline{K}/K_{\Sigma})$ acts trivially on $A_1$ and $A_2$ and we can interpret $\phi$ as a $\Gal(K_{\Sigma}/K)$-isomorphism. Then $\phi$ induces an isomorphism
  \[ \phi \colon H^1(K_\Sigma/L, A_1[\pi^i]) \longrightarrow H^1(K_\Sigma/L, A_2[\pi^i])\] 
  of $G_K$-modules. 
  
  For any prime $v$ of $L$, the inclusion $ G_{L_v} \hookrightarrow G_{L}$ of the local absolute Galois group at the completion $L_v$ of $L$ at $v$ induces an isomorphism \[H^1(L_v, A_1[\pi^i]) \longrightarrow H^1(L_v, A_2[\pi^i]). \] 
  The corresponding isomorphism between fine Selmer groups is now immediate. 
\end{proof}

Now we turn to the proof of our first main result. 
\begin{proof}[Proof of Theorem~\ref{thm:A}]
Let $l$ be such that $(\pi^lY_{A_1,\Sigma}^{(K_{\infty})})[\pi]$ is pseudo-null. 
By definition of the $\mu$-invariant (see \eqref{eq:mu}) and Lemma~\ref{lemma:mu}, we have 
\begin{align}
    \label{invariants}
\mu(Y_{A_1,\Sigma}^{(K_\infty)})&=\sum_{i=0}^{\infty} (\rg_{\F_q[[G]]}(\pi^iY_{A_1,\Sigma}^{(K_\infty)}/\pi^{i+1}Y_{A_1,\Sigma}^{(K_\infty)}) - r_1) \nonumber 
\\&=\sum_{i=0}^{l-1}  (\rg_{\F_q[[G]]}(\pi^iY_{A_1,\Sigma}^{(K_\infty)}/\pi^{i+1}Y_{A_1,\Sigma}^{(K_\infty)}) - r_1).\end{align}
Now Corollary~\ref{cor:vergleich} implies that for both $j = 1$ and $j = 2$, the $\F_q[[G]]$-rank of $\pi^iY_{A_j,\Sigma}^{(K_{\infty})}/\pi^{i+1}Y_{A_j,\Sigma}^{(K_{\infty})}$ equals 
\[ \rg_{\F_q[[G]]}(\varprojlim_{K \subseteq L \subseteq K_\infty} \textup{Sel}_{0,A_j[\pi^{i+1}],\Sigma}(L)^\vee/ \varprojlim_{K \subseteq L \subseteq K_\infty} \textup{Sel}_{0,A_j[\pi^i],\Sigma}(L)^\vee). \] 
 Using that $A_1[\pi^l]\cong A_2[\pi^l]$, Lemma~\ref{lemma:G-iso} implies that  $\textup{Sel}_{0,A_1[\pi^i],\Sigma}(L)\cong \textup{Sel}_{0,A_2[\pi^i],\Sigma}(L)$ for every $i \le l$. By  \eqref{invariants}, we may conclude that 
 \begin{align*}
\mu(Y_{A_1,\Sigma}^{(K_\infty)})&=\sum_{i=0}^{l-1} (\rg_{\F_q[[G]]}(\pi^iY_{A_1,\Sigma}^{(K_\infty)}/\pi^{i+1}Y_{A_1,\Sigma}^{(K_\infty)}) - r_1)\\
&=\sum_{i=0}^{l-1} (\rg_{\F_q[[G]]}(\pi^iY_{A_2,\Sigma}^{(K_\infty)}/\pi^{i+1}Y_{A_2,\Sigma}^{(K_\infty)}) - r_1)\\
&\le \sum_{i=0}^{\infty} (\rg_{\F_q[[G]]}(\pi^iY_{A_2,\Sigma}^{(K_\infty)}/\pi^{i+1}Y_{A_2,\Sigma}^{(K_\infty)}) - r_2) 
= \mu(Y_{A_2,\Sigma}^{(K_\infty)}). \end{align*} 
Here we used the hypothesis $r_2 \le r_1$, i.e. $-r_1 \le -r_2$. 

Now we prove assertion (b). In the following, we abbreviate $\rg_{\F_q[[G]]}$ to $r$. If $A_1[\pi^{l+1}]\cong A_2[\pi^{l+1}]$ then \begin{align*}
r(\pi^{l} Y_{A_1,\Sigma}^{(K_{\infty})}/\pi^{l+1}Y_{A_1,\Sigma}^{(K_{\infty})}) & = r(\varprojlim_{K \subseteq L \subseteq K_\infty} \textup{Sel}_{0,A_1[\pi^{l+1}],\Sigma}(L)^\vee / \varprojlim_{K \subseteq L \subseteq K_\infty}\textup{Sel}_{0,A_1[\pi^{l}],\Sigma}(L)^\vee)\\
& = r(\varprojlim_{K \subseteq L \subseteq K_\infty}\textup{Sel}_{0,A_2[\pi^{l+1}],\Sigma}(L)^\vee / \varprojlim_{K \subseteq L \subseteq K_\infty} \textup{Sel}_{0,A_2[\pi^{l}],\Sigma}(L)^\vee)\\
& = r(\pi^{l}Y_{A_2,\Sigma}^{(K_{\infty})}/\pi^{l+1}Y_{A_2,\Sigma}^{(K_{\infty})}).\end{align*}
Using \cite[Proposition~4.12]{lim_fuer_howson} and the definition of $l$, we obtain 
\begin{eqnarray*} 
  r(\pi^l Y_{A_1,\Sigma}^{(K_\infty)}/\pi^{l+1}Y_{A_1,\Sigma}^{(K_\infty)}) & = & r((\pi^lY_{A_1,\Sigma}^{(K_\infty)})[\pi]) + \rg_{\Z_q[[G]]}(\pi^l Y_{A_1,\Sigma}^{(K_\infty)}) \\ 
  & = & 0 + r_1, 
\end{eqnarray*} 
and similarly 
$$ r(\pi^l Y_{A_2,\Sigma}^{(K_\infty)}/\pi^{l+1}Y_{A_2,\Sigma}^{(K_\infty)}) \ge r_2. $$ 
This proves the first claim of (b). 

If $r_2 = r_1$, then it follows from the above that ${\rg_{\F_q[[G]]}((\pi^lY_{A_2,\Sigma}^{(K_\infty)})[\pi]) = 0}$. In view of Remark \ref{finitesum} this implies that  
$$\mu(Y_{A_2,\Sigma}^{(K_\infty)}) = \sum_{i=0}^{l-1} (\rg_{\F_q[[G]]}(\pi^i Y_{A_2,\Sigma}^{(K_\infty)}/\pi^{i+1}Y_{A_2,\Sigma}^{(K_\infty)}) - r_2) = \mu(Y_{A_1,\Sigma}^{(K_\infty)}), $$ proving the second claim of (b).  
Using the equality of ranks derived above we obtain that 
\[\rg_{\F_q[[G]]}(\pi^iY_{A_1,\Sigma}^{(K_{\infty})}/\pi^{i+1}Y_{A_1,\Sigma}^{(K_{\infty})}) = \rg_{\F_q[[G]]}(\pi^iY_{A_2,\Sigma}^{(K_{\infty})}/\pi^{i+1}Y_{A_2,\Sigma}^{(K_{\infty})})\] 
for all $0\le i\le l$. Let $E_j = \bigoplus_{i=1}^{s_j} \mathcal{O}[[G]]/(\pi^{e_i^j})$ be the elementary $\mathcal{O}[[G]]$-module associated to $Y_{A_j,\Sigma}^{(K_{\infty})}[\pi^{\infty}]$ via \cite[Theorem~3.40]{venjakob} (Venjakob's result is proven only for $\Ok = \Z_p$, but it is valid in our more general setting) and define  $$f_i^{j}=  \vert \{k\mid e_k^j>i\}\vert.$$ Since $r_1=r_2 =:r$ we obtain
\[\rg_{\F_q[[G]]}(\pi^iY_{A_j,\Sigma}^{(K_{\infty})}/\pi^{i+1}Y_{A_j,\Sigma}^{(K_{\infty})}) = r+f^j_i \] 
by Lemma~\ref{lemma:mu}. 
Therefore, $f^1_i=f^2_i$ for all $0\le i\le l$ and thus $E_1=E_2$ from which the claim is immediate.

The assertion (c) is a special case of (b). Finally, if $(\pi^l Y_{A_2,\Sigma}^{(K_\infty)})[\pi]$ is also pseudo-null, then we can exchange the roles of $A_1$ and $A_2$ and obtain equality of $\mu$-invariants. 
\end{proof} 

\subsection{The completely split case}
\label{sec:split}
Now we treat admissible $p$-adic Lie extensions $K_\infty/K$ such that some $v \in \Sigma$ may be completely split in $K_\infty/K$. In this case, we work under the restrictive assumption that $A_j(K_v)[\pi] = \{0\}$ for every $v \in \Sigma$ and $j\in \{1,2\}$ (meaning that the subgroup of $A_j[\pi]$ fixed by $G_{K_v} \subseteq G_K$ is trivial). First, we derive several auxiliary results, starting with a lemma which will serve as a substitute for Lemma \ref{lemma:vergleich}. 
\begin{lemma} 
\label{lemma:isomorphisselmergroups} 
  Let $A$ be associated with a $p$-adic $G_K$-representation $V$ and let $\Sigma$ be a finite set of primes of $K$ containing $S_p \cup S_{\textup{ram}}(A)$. If $p = 2$, then we assume that $K$ is totally imaginary. Assume that $A(K_v)[p] = \{0\}$ for every $v \in \Sigma$. 
  Then
  \[\Sel_{0,A,\Sigma}(L)[\pi^i]\cong \Sel_{0,A[\pi^i],\Sigma}(L) \] 
  for every finite normal $p$-extension $L \subseteq K_\Sigma$ of $K$ and each $i \in \N$, $i \ge 1$. 
  \end{lemma} 
\begin{proof}
The assumptions imply that $A(K)[\pi]=\{ 0 \}$. Since $L/K$ is a $p$-extension, it follows from \cite[Corollary~(1.6.13)]{nsw} that also $H^0(L, A[\pi]) = A(L)[\pi]=\{0\}$. Hence $H^0(L,A)= \{ 0 \}$. As $V$ is unramified outside of $\Sigma$ 
we obtain that $H^0(K_{\Sigma}/L, A)=0$.
Now consider the exact sequence 
\[0\longrightarrow A[\pi^i]\longrightarrow A\stackrel{ \cdot \pi^i}{\longrightarrow} A\longrightarrow 0.\]
Taking $K_{\Sigma}/L$-cohomology we obtain a second exact sequence
\[0\longrightarrow H^1(K_{\Sigma}/L,A[\pi^i])\longrightarrow H^1(K_{\Sigma}/L, A)\longrightarrow H^1(K_{\Sigma}/L,A), \]
where the last homomorphism is multiplication by $\pi^i$. Hence, we obtain the isomorphism
\[H^1(K_{\Sigma}/L,A[\pi^i])\cong H^1(K_{\Sigma}/L,A)[\pi^i].\]
Let now $w$ be a place in $L$ above a prime $v \in \Sigma$.
Using analogous arguments we can derive from the hypothesis $A(K_v)[\pi] = \{0\}$ that
\[H^1(L_{w},A[\pi^i])\cong H^1(L_{w},A)[\pi^i].\] 
\end{proof}
\begin{cor} \label{cor:isomorphisselmergroups} 
   Let $A$ be as above, let $K_\infty/K$ be an admissible $p$-adic Lie extension, and let $\Sigma$ be a finite set of primes of $K$ containing ${S_{\textup{ram}}(K_\infty/K) \cup S_p \cup  S_{\textup{ram}}(A)}$. If $p = 2$, then we assume that $K$ is totally imaginary. Assume that $A(K_v)[\pi] = \{0\}$ for every $v \in \Sigma$. 
  Then
  \[ Y_{A,\Sigma}^{(K_\infty)} / \pi^i Y_{A,\Sigma}^{(K_\infty)} \cong \varprojlim_{K \subseteq L \subseteq K_\infty}  \Sel_{0,A[\pi^i],\Sigma}(L)^\vee \] 
  for every $i \in \N$, $i \ge 1$, where the projective limit is taken over the finite normal subextensions $L$ of $K_\infty/K$. 
\end{cor} 
\begin{proof} 
   In view of Lemma~\ref{lemma:isomorphisselmergroups}, we have isomorphisms 
   \[ \Sel_{0,A,\Sigma}(L)[\pi^i]\cong \Sel_{0,A[\pi^i],\Sigma}(L)\] 
   for each $L$. By duality $\pi^{i}Y_{A,\Sigma}^{(L)}$ is precisely the group acting trivial on $\Sel_{0,A,\Sigma}(L)[\pi^i]$. Therefore 
   \[ Y_{A,\Sigma}^{(L)} / \pi^i Y_{A,\Sigma}^{(L)} \cong (\Sel_{0,A,\Sigma}(L)[\pi^i])^\vee \cong \Sel_{0,A[\pi^i],\Sigma}(L)^\vee. \] The result now follows since 
   \[ Y_{A,\Sigma}^{(K_\infty)} / \pi^i Y_{A,\Sigma}^{(K_\infty)} \cong \varprojlim_{K \subseteq L \subseteq K_\infty} Y_{A,\Sigma}^{(L)} / \pi^i Y_{A,\Sigma}^{(L)}. \] 
\end{proof} 

We can now prove an analogon of Theorem~\ref{thm:A}: 
\begin{thm} \label{thm:A2} 
   Let $A_1$ and $A_2$ be associated with two representations $V_1$ and $V_2$ of $G_K$. Let $K_\infty/K$ be an admissible $p$-adic Lie extension, and let $G = \Gal(K_\infty/K)$. Let $\Sigma$ be a finite set of primes of $K$ which contains $S_{\text{ram}}(K_\infty/K)$, $S_p$ and the sets of primes of $K$ where either $V_1$ or $V_2$ is ramified. If $p = 2$, then we assume that $K$ is totally imaginary. 
   
   Suppose that $A_j(K_v)[\pi] = \{0\}$ for every $v \in \Sigma$ and $j\in \{1,2\}$. We let $r_j = \rg_{\mathcal{O}[[G]]}(Y_{A_j,\Sigma}^{(K_\infty)})$, $1 \le j \le 2$. Let $l$ be minimal such that $(\pi^l Y_{A_1,\Sigma}^{(K_\infty)})[\pi]$ is pseudo-null. 
   
   Then the statements from Theorem~\ref{thm:A} hold. 
\end{thm} 
\begin{proof} 
  Suppose that $A_1[\pi^l] \cong A_2[\pi^l]$ as $G_K$-modules. Then  
  \begin{eqnarray*} \pi^{i} Y_{A_1,\Sigma}^{(K_{\infty})}/\pi^{i+1}Y_{A_1,\Sigma}^{(K_{\infty})} & \cong & \varprojlim_{K \subseteq L \subseteq K_\infty} \textup{Sel}_{0,A_1[\pi^{i+1}],\Sigma}(L)^\vee / \varprojlim_{K \subseteq L \subseteq K_\infty}\textup{Sel}_{0,A_1[\pi^{i}],\Sigma}(L)^\vee \\ 
  & \cong & \varprojlim_{K \subseteq L \subseteq K_\infty} \textup{Sel}_{0,A_2[\pi^{i+1}],\Sigma}(L)^\vee / \varprojlim_{K \subseteq L \subseteq K_\infty}\textup{Sel}_{0,A_2[\pi^{i}],\Sigma}(L)^\vee \\ 
  & \cong & \pi^{i} Y_{A_2,\Sigma}^{(K_{\infty})}/\pi^{i+1}Y_{A_2,\Sigma}^{(K_{\infty})}
  \end{eqnarray*} 
  for every $i < l$, by Corollary~\ref{cor:isomorphisselmergroups} and Lemma~\ref{lemma:G-iso}. It follows in particular that both $\F_q[[G]]$-modules have the same rank. Therefore we can proceed as in the proof of Theorem~\ref{thm:A}. 
\end{proof}

\begin{proof}[Proof of Theorem~\ref{thm:B}] 
The hypothesis in (a) implies that $\pi^l Y_{A_1,\Sigma}^{(K_\infty)}$ is $\mathcal{O}$-free. Let $E$ be an elementary $\Lambda$-module pseudo-isomorphic to $\pi^l Y_{A_1,\Sigma}^{(K_\infty)}$. Since $\pi^l Y_{A_1,\Sigma}^{(K_\infty)}$ is $\mathcal{O}$-free, the maximal finite submodule of $Y_{A_1,\Sigma}^{(K_\infty)}$ is annihilated by $\pi^l$. For any finitely generated $\mathcal{O}$-module $M$ we denote by $\rg_q(M)$ the dimension of $M/\pi M$ as $\F_q$-vector space. Then ${\rg_q(\pi^l Y_{A_1, \Sigma}^{(K_\infty)}) = \rg_q(E)}$ by the argument given in the proof of \cite[Proposition~3.4(i)]{local_beh}: as $\pi^l Y_{A_1,\Sigma}^{(K_\infty)}$ is $\Ok$-free, we have an injection ${\varphi \colon \pi^l Y_{A_1,\Sigma}^{(K_\infty)} \longrightarrow E}$ with finite cokernel. Moreover, since multiplication by $\pi$ is injective on $E$, the quotients $E/\textup{im}(\varphi)$ and $\pi E / \pi \textup{im}(\varphi)$ are isomorphic, proving that $$\rg_q(\pi^l Y_{A_1, \Sigma}^{(K_\infty)}) = \rg_q(\textup{im}(\varphi))$$ 
equals $\rg_q(E)$. 

Therefore  
\[ |\pi^{l} Y_{A_1,\Sigma}^{(K_{\infty})}/\pi^{l+1}Y_{A_1,\Sigma}^{(K_{\infty)}}| = q^{\rg_q(\pi^lY_{A_1,\Sigma}^{(K_\infty)})} = q^{\rg_q(E)} = q^{\lambda(E)} = q^{\lambda(Y_{A_1,\Sigma}^{(K_\infty)})}. \]
On the other hand, the maximal finite $\Lambda$-submodule of $Y_{A_2,\Sigma}^{(K_\infty)}$ need not be annihilated by $\pi^l$; therefore 
\[ |\pi^{l} Y_{A_2,\Sigma}^{(K_{\infty})}/\pi^{l+1}Y_{A_2, \Sigma},\Sigma^{(K_{\infty)}}| \ge q^{\lambda( Y_{A_2,\Sigma}^{(K_\infty)})}. \]

If both $\pi^l Y_{A_1,\Sigma}^{(K_\infty)}$ and $\pi^l Y_{A_2,\Sigma}^{(K_\infty)}$ are $\mathcal{O}$-free, then we can exchange the roles of $A_1$ and $A_2$ and obtain equality of $\lambda$-invariants. 
\end{proof}

\section{Non $p$-extensions} \label{section:q} 
In this final section we study the growth of fine Selmer groups of congruent Galois representations over abelian algebraic extensions of $K$ which are the compositum of finite $r$-extensions for suitable primes $r \ne p$. If $p = 2$, then we always assume that $K$ is totally imaginary. Similarly as in Sections~\ref{sec:generic} and \ref{sec:split}, we distinguish between two different settings, starting with one resembling the case which has been studied in Section~\ref{sec:split}. 
\begin{thm} \label{thm:q} 
  Let $p$ be a fixed prime, let $A_1$ and $A_2$ be associated with two representations of $G_K$, and let $K_\infty/K$ be an abelian algebraic extension. Let $\mathcal{P}$ be the set of primes $r$ such that $K_\infty/K$ contains a finite subextension of degree $r$ over $K$. 
  Let $\Sigma$ be a finite set of primes of $K$ which contains $S_p$, $S_{ram}(K_\infty/K)$ and $S_{\textup{ram}}(A_j)$, $j = 1,2$. 
  
  We assume that ${\dim(A_1) = \dim(A_2) =: d}$, that $r \ge q^{d}$ for each $r \in \mathcal{P}$, and that ${A_j(K_v)[\pi] = \{0\}}$ for $j = 1,2$ and every $v \in \Sigma$. 
  Then the following statements hold: \begin{compactenum}[(a)] 
  \item If $A_1[\pi] \cong A_2[\pi]$ as $G_K$-modules, then $Y_{A_1, \Sigma}^{(K_\infty)}$ is a finitely generated $\Ok$-module if and only if $Y_{A_2,\Sigma}^{(K_\infty)}$ is finitely generated over $\Ok$.
  \item Suppose that both $Y_{A_j,\Sigma}^{(K_\infty)}$, $j = 1,2$, are finitely generated over $\Ok$. Let $l \in \N$ be large enough such that $(\pi^l Y_{A_1, \Sigma}^{(K_\infty)})[\pi] = \{0\}$. If $A_1[\pi^{l+1}] \cong A_2[\pi^{l+1}]$ as $G_K$-modules, then $\rg_{\Ok}(Y_{A_2,\Sigma}^{(K_\infty)}) \le \rg_{\Ok}(Y_{A_1,\Sigma}^{(K_\infty)})$. 
  \item In the setting of (b), suppose that $A_1[\pi^{l+1}] \cong A_2[\pi^{l+1}]$ for some $l$ such that both $(\pi^l Y_{A_j,\Sigma}^{(K_\infty)})[\pi]$ are trivial. Then $\rg_{\Ok}(Y_{A_2,\Sigma}^{(K_\infty)}) =  \rg_{\Ok}(Y_{A_1,\Sigma}^{(K_\infty)})$. 
  \end{compactenum} 
\end{thm} 
\begin{proof} 
The proof is analogous to the proofs of Theorems~\ref{thm:A2} and \ref{thm:B}. Instead of Lemma~\ref{lemma:isomorphisselmergroups}, we apply the following 
\begin{lemma} 
\label{lemma:fuer_q} 
  Let $A$ be associated with a $G_K$-representation of dimension $d$, let $\Sigma$ be a finite set of primes of $K$ containing ${S_p \cup S_{\textup{ram}}(A)}$. If $p = 2$, then we assume that $K$ is totally imaginary. Assume that $A(K_v)[\pi] = \{0\}$ for every $v \in \Sigma$. 
  
  Let $L \subseteq K_\Sigma$ be a finite normal extension of $K$ such that each prime number $r$ dividing $[L:K]$ satisfies $r \ge q^{d}$. Then 
  \[\Sel_{0,A,\Sigma}(L)[\pi^i]\cong \Sel_{0,A[\pi^i],\Sigma}(L) \] 
  for each $i \in \N$, $i \ge 1$. 
\end{lemma} 
\begin{proof}
By assumption $A(K)[\pi]=\{ 0 \}$. We mimic the proof of \cite[Corollary~(1.6.13)]{nsw} and show that also $H^0(L, A[\pi]) = A(L)[\pi]=\{0\}$. 
Let $r$ be the smallest prime number dividing $[L:K]$. Since $A(L)[\pi] \setminus A(K)[\pi]$ is the disjoint union of $\Gal(L/K)$-orbits with more than one element, the cardinality of each such orbit is divisible by some prime $r' \ge r$. Thus if there exists at least one orbit containing more than one element, then 
$$|A(L)[\pi]| \ge |A(K)[\pi]|+r'\ge |A(K)[\pi]|+r. $$
On the other hand, $|A(L)[\pi]| \le q^{d}$. Since $r \ge q^{d}$ by assumption, we obtain that such a non-trivial orbit cannot exist. Therefore $A(L)[\pi]=\{0\}$. 
Now we can proceed as in the proof of Lemma~\ref{lemma:isomorphisselmergroups}. 
\end{proof} 
Using this lemma, we can derive a variant of Corollary~\ref{cor:isomorphisselmergroups} for all finite normal subextensions of $K_\infty$. Then we can proceed as in the proofs of Theorems~\ref{thm:A} and \ref{thm:B}. 
\end{proof} 

Now we turn to the second result for non-$p$-extensions. As in Theorem~\ref{thm:q}, we let $\mathcal{P} = \mathcal{P}(K_\infty)$ be the set of prime numbers $r$ such that $K_\infty$ contains an extension of $K$ of degree $r$. 
\begin{thm} \label{thm:q2} 
  Let $p$ be a fixed prime, let $A_1$ and $A_2$ be associated with two $G_K$-representations, and let $K_\infty/K$ be an abelian algebraic extension such that ${p \not\in \mathcal{P}(K_\infty)}$. 
  Let $\Sigma$ be a finite set of primes of $K$ which contains $S_p$, $S_{ram}(K_\infty/K)$ and $S_{\textup{ram}}(A_j)$, $j = 1,2$. 
  
  We assume that each prime $v \in \Sigma$ is finitely split in $K_\infty/K$. 
  \begin{compactenum}[(a)] 
  \item If $A_1[\pi] \cong A_2[\pi]$ as $G_K$-modules, then $Y_{A_1,\Sigma}^{(K_\infty)}$ is a finitely generated $\Ok$-module if and only if $Y_{A_2,\Sigma}^{(K_\infty)}$ is finitely generated over $\Ok$. 
  \end{compactenum} 
  Suppose now that for each $j\in\{1,2\}$ and every $w \in \Sigma(K_\infty)$, the group $A_j(K_{\infty,w})[\pi^\infty]$ is finite. Then also the following statements hold: \begin{compactenum} 
  \item[(b)] Suppose that both $Y_{A_j,\Sigma}^{(K_\infty)}$, $j = 1,2$, are finitely generated over $\Ok$. Let $l \in \N$ be large enough such that $(\pi^l Y_{A_1,\Sigma}^{(K_\infty)})[\pi] = \{0\}$ and $\pi^l A_1(K_{\infty,w})[\pi^\infty] = \{0\}$ for every $w \in \Sigma(K_\infty)$. If $A_1[\pi^{l+1}] \cong A_2[\pi^{l+1}]$ as $G_K$-modules, then $\rg_{\Ok}(Y_{A_2,\Sigma}^{(K_\infty)}) \le \rg_{\Ok}(Y_{A_1,\Sigma}^{(K_\infty)})$. 
  \item[(c)] In the setting of (b), if $A_1[\pi^{l+1}] \cong A_2[\pi^{l+1}]$ for some $l$ such that both $(\pi^l Y_{A_j,\Sigma}^{(K_\infty)})[\pi]$ and all the groups $\pi^l A_j(K_{\infty,w})[\pi^\infty]$, $j \in \{1,2\}$, are trivial, then $$\rg_{\Ok}(Y_{A_2,\Sigma}^{(K_\infty)}) =  \rg_{\Ok}(Y_{A_1,\Sigma}^{(K_\infty)}). $$ 
  \end{compactenum} 
\end{thm} 

\begin{proof} 
  We first note that $Y_{A_j,\Sigma}^{(K_\infty)}$ is finitely generated over $\Ok$ if and only if the quotient  $Y_{A_j,\Sigma}^{(K_\infty)}/\pi Y_{A_j,\Sigma}^{(K_\infty)}$ is finite. The exact sequence \eqref{eq:vergleich} from the proof of Corollary~\ref{cor:vergleich} implies that the kernels and cokernels of the maps 
  \[ Y_{A_j,\Sigma}^{(K_\infty)}/\pi Y_{A_j,\Sigma}^{(K_\infty)} \longrightarrow \varprojlim_{K \subseteq L \subseteq K_\infty} \textup{Sel}_{0,A_j[\pi],\Sigma}(L)^\vee\] 
  are finite for both $j = 1$ and $j = 2$ (here we use the hypothesis that each $v \in \Sigma$ is finitely split in $K_\infty/K$). The assertion (a) therefore follows from arguments similar to those used in the proof of Theorem~\ref{thm:A}. 
  
  More generally, we have exact sequences 
  \begin{align*} 0 \longrightarrow M_k \longrightarrow Y_{A_j,\Sigma}^{(K_\infty)}/\pi^k Y_{A_j,\Sigma}^{(K_\infty)} \longrightarrow \varprojlim_{K \subseteq L \subseteq K_\infty} \Sel_{0,A_j[\pi^k],\Sigma}(L)^\vee \longrightarrow N_k \longrightarrow 0 \end{align*} 
  for finite abelian groups $M_k$ and $N_k$, $k \in \N$. Moreover, from the proof of Lemma~\ref{lemma:vergleich}, we obtain exact sequences 
  \begin{align}
 \label{eq:M_k} 0 \longrightarrow &N_k \longrightarrow A_j(K_\infty)[\pi^\infty]/\pi^k A_j(K_\infty)[\pi^\infty] \longrightarrow C_k \longrightarrow 0, \\ \nonumber & 0 \longrightarrow C_k \longrightarrow\bigoplus_{w \in \Sigma(K_\infty)} A_j(K_{\infty,w})[\pi^\infty]/\pi^k A_j(K_{\infty,w})[\pi^\infty] \longrightarrow M_k  \longrightarrow 0  \end{align}
  for every $k \in \N$ and $j = 1,2$ (note that $M_k \cong M_k^\vee$ and $N_k \cong N_k^\vee$). Therefore 
  $|\pi^k Y_{A_j, \Sigma}^{(K_\infty)}/\pi^{k+1}Y_{A_j, \Sigma}^{(K_\infty)}|$ differs from  
   \[|\varprojlim_{K \subseteq L \subseteq K_\infty} \Sel_{0,A_j[\pi^{k+1}],\Sigma}(L)^\vee / \varprojlim_{K \subseteq L \subseteq K_\infty} \Sel_{0,A_j[\pi^k],\Sigma}(L)^\vee|\] 
   by a factor $\frac{|M_k| |N_{k+1}|}{|N_k||M_{k+1}|}$ which is smaller than or equal to 1 because \eqref{eq:M_k} implies that $\frac{|N_{k+1}|}{|N_k|} \le \frac{|M_{k+1}|}{|M_k|}$. In fact, for $k = l$ and $j = 1$, this factor is 1 by our hypotheses. Therefore 
   \[ |\pi^l Y_{A_1,\Sigma}^{(K_\infty)}/\pi^{l+1}Y_{A_1,\Sigma}^{(K_\infty)}| = |\varprojlim_{K \subseteq L \subseteq K_\infty} \Sel_{0,A_2[\pi^{l+1}],\Sigma}(L)^\vee / \varprojlim_{K \subseteq L \subseteq K_\infty} \Sel_{0,A_2[\pi^l],\Sigma}(L)^\vee|\]
   because $A_1[\pi^{l+1}] \cong A_2[\pi^{l+1}]$. Note that the factor $\frac{|M_k| |N_{k+1}|}{|N_k||M_{k+1}|}$ can be strictly smaller than 1 for $A_2$. This happens if $\pi^l$ does not annihilate the $\pi$-primary subgroups of the $A_2(K_{\infty,w}$). We have thus shown that $$\rg_q(\pi^l Y_{A_2,\Sigma}^{(K_\infty)}) \le \rg_q(\pi^l Y_{A_1,\Sigma}^{(K_\infty)}),$$ 
   where $\rg_q$ is defined as in the proof of Theorem~\ref{thm:B}. 
   The assertion (b) follows since 
   \[ |\pi^l Y_{A_1,\Sigma}^{(K_\infty)}/\pi^{l+1}Y_{A_1,\Sigma}^{(K_\infty)}| = q^{\rg_{\Ok}(Y_{A_1,\Sigma}^{(K_\infty)})} \] 
   because $(\pi^l Y_{A_1,\Sigma}^{(K_\infty)})[\pi] = \{0\}$ by assumption; the $\Ok$-rank of $Y_{A_2,\Sigma}^{(K_\infty)}$ can be strictly smaller than the corresponding $q$-rank, as in the proof of Theorem~\ref{thm:B}. 
   
   Finally, (c) follows by reverting the roles of $A_1$ and $A_2$ in the previous proof. 
  \end{proof} 
\begin{rem}
  Going through the proof of the theorem, one sees that actually the finiteness of $A_2(K_{\infty,w})[\pi^\infty]$ is needed only for at least one $w \in \Sigma(K_\infty)$. Moreover, if one assumes that $A(K_{\infty})[\pi]=\{0\}$, then one can drop completely the condition that the group $A_2(K_{\infty,w})[\pi^\infty]$ is finite for every $w \in \Sigma(K_\infty)$ in point (b) of the above theorem. 
\end{rem}

\begin{rem} 
  In order to give some evidence for the finiteness assumptions in the last two parts of Theorem~\ref{thm:q2}, we mention some known results in the special setting of abelian varieties. In the following, we let $A$ be an abelian variety defined over the number field $K$, and we consider $\Ok = \Z_p$. 
  
  Actually the following conditions are sufficient for ensuring finite torsion groups, i.e. not only finite $p$-torsion for some fixed prime $p$. \begin{compactenum}[(i)] 
    \item If $K_\infty/K$ is a finite extension, then the torsion subgroup of $A(K_{\infty,w})$ is finite for each prime $w$ by the theorem of Mattuck (see \cite{mattuck}). 
    \item If $A$ has potentially good and ordinary reduction at some prime $q$, then the torsion subgroup of $A(K_{\infty,w})$ is finite for each $w \mid q$ if $K_\infty$ is a finite extension of the cyclotomic $\Z_q$-extension of $K$ (see \cite{imai}). 
    \item For global fields, more is known: let $\Omega$ be the field obtained from $K$ by adjoining \emph{all} roots of unity in some fixed algebraic closure of $K$ (i.e. $\Omega$ contains the cyclotomic $\Z_q$-extensions for all primes $q$). Then it follows from results of Ribet (see \cite[Appendix, Theorem~1]{ribet}) that the torsion group of $\Omega(A)$ is finite. 
  \end{compactenum} 
\end{rem} 

We conclude by mentioning a special setting, namely of an elliptic curve $A = E$ defined over $K$, in which $Y_{A, \Sigma}^{(K_\infty)}$ is known to be finitely generated over $\Ok = \Z_p$ for an infinite non-$p$-extension $K_\infty$ of $K$. 
\begin{example} 
  Let $N$ be an imaginary quadratic number field, and let $E$ be an elliptic curve with complex multiplication by the ring of integers $\mathcal{O}_N$ of $N$. Let $q > 3$ be a prime of good reduction which splits in $N$, $q \mathcal{O}_N = \q \overline{\q}$. Let $K$ be an abelian extension of $N$ which is tamely ramified at $\q$, and let $K_\infty = K \cdot  N(E[\q^\infty])$. 
  
  Now suppose that $p \ne q$ is a prime number which is co-prime with $6 [K:N]$. We assume that $p$ splits in $N/\Q$, does not ramify in $K/N$ and that $E$ has good reduction at the primes of $N$ above $p$. Let $\Sigma$ be a finite set of primes of $K$ which contains $S_{\textup{ram}}(K_\infty/K)$, $S_p$ and $S_{\textup{ram}}(E)$. If $E(K)[p] = \{0\}$ and $E(K_v)[p] = \{0\}$ for every $v \in \Sigma$, then 
  $\rg_{\Z_p}(Y_{E',\Sigma}^{(K_\infty)})$ is finite for every elliptic curve $E'$ which is defined over $K$ and satisfies $S_{\textup{ram}}(E') \subseteq \Sigma$ and $E'[p] \cong E[p]$ as $G_K$-modules. 

   Indeed, by \cite[Theorem~1.2]{lamplugh}, the hypotheses of the corollary imply that in fact $X_E^{(K_\infty)}$ is finitely generated over $\Z_p$. Now we can apply Theorem~\ref{thm:q}. 
\end{example}

\bibliography{references} 

\newcommand{\etalchar}[1]{$^{#1}$}
\begin{thebibliography}{CFK{\etalchar{+}}05}

\bibitem[AS15]{ahmed-shekhar}
S.~Ahmed and S.~Shekhar.
\newblock {$\lambda$}-invariants of {S}elmer groups of elliptic curves with
  positive {$\mu$}-invariant.
\newblock {\em J. Ramanujan Math. Soc.}, 30(1):115--133, 2015.

\bibitem[Bar13]{barth}
P.~Barth.
\newblock Iwasawa theory for one-parameter families of motives.
\newblock {\em Int. J. Number Theory}, 9(2):257--319, 2013.

\bibitem[BS10]{barman-saikia}
R.~{Barman} and A.~{Saikia}.
\newblock {A note on Iwasawa \(\mu\)-invariants of elliptic curves}.
\newblock {\em {Bull. Braz. Math. Soc. (N.S.)}}, 41(3):399--407, 2010.

\bibitem[CFK{\etalchar{+}}05]{5_people}
J.~Coates, T.~Fukaya, K.~Kato, R.~Sujatha, and O.~Venjakob.
\newblock The {$\rm GL_2$} main conjecture for elliptic curves without complex
  multiplication.
\newblock {\em Publ. Math. Inst. Hautes \'{E}tudes Sci.}, (101):163--208, 2005.

\bibitem[CH01]{coates-howson}
J.~H. Coates and S.~Howson.
\newblock Euler characteristics and elliptic curves. {II}.
\newblock {\em J. Math. Soc. Japan}, 53(1):175--235, 2001.

\bibitem[EPW06]{emerton-pollack-weston}
M.~Emerton, R.~Pollack, and T.~Weston.
\newblock Variation of {I}wasawa invariants in {H}ida families.
\newblock {\em Invent. Math.}, 163(3):523--580, 2006.

\bibitem[Gre89]{greenberg89}
R.~Greenberg.
\newblock Iwasawa theory for {$p$}-adic representations.
\newblock In {\em Algebraic number theory}, volume~17 of {\em Adv. Stud. Pure
  Math.}, pages 97--137. Academic Press, Boston, MA, 1989.

\bibitem[GV00]{greenberg-vatsal}
R.~Greenberg and V.~Vatsal.
\newblock On the {I}wasawa invariants of elliptic curves.
\newblock {\em Invent. Math.}, 142(1):17--63, 2000.

\bibitem[GW04]{goodearl-warfield}
K.~R. Goodearl and R.~B. Warfield, Jr.
\newblock {\em An introduction to noncommutative {N}oetherian rings}, volume~61
  of {\em London Mathematical Society Student Texts}.
\newblock Cambridge University Press, Cambridge, second edition, 2004.

\bibitem[Hac11]{hachimori}
Y.~Hachimori.
\newblock Iwasawa {$\lambda$}-invariants and congruence of {G}alois
  representations.
\newblock {\em J. Ramanujan Math. Soc.}, 26(2):203--217, 2011.

\bibitem[Hat17]{hatley}
J.~Hatley.
\newblock Rank parity for congruent supersingular elliptic curves.
\newblock {\em Proc. Amer. Math. Soc.}, 145(9):3775--3786, 2017.

\bibitem[HL19]{hatley-lei}
J.~Hatley and A.~Lei.
\newblock Comparing anticyclotomic {S}elmer groups of positive coranks for
  congruent modular forms.
\newblock {\em Math. Res. Lett.}, 26(4):1115--1144, 2019.

\bibitem[How02]{howson}
S.~Howson.
\newblock {Euler characteristics as invariants of {I}wasawa modules}.
\newblock {\em Proc. London Math. Soc. (3)}, 85(3):634--658, 2002.

\bibitem[{Ima}75]{imai}
H.~{Imai}.
\newblock {A remark on the rational points of Abelian varieties with values in
  cyclotomic {$\mathbb{Z}_p$}-extensions.}
\newblock {\em {Proc. Japan Acad.}}, 51:12--16, 1975.

\bibitem[Jha12]{jha}
S.~Jha.
\newblock Fine {S}elmer group of {H}ida deformations over non-commutative
  {$p$}-adic {L}ie extensions.
\newblock {\em Asian J. Math.}, 16(2):353--365, 2012.

\bibitem[KL81]{ribet}
N.~M. {Katz} and S.~{Lang}.
\newblock {Finiteness theorems in geometric classfield theory. (With an
  appendix by K. A. Ribet).}
\newblock {\em {Enseign. Math. (2)}}, 27:285--314, 315--319, 1981.

\bibitem[{Kle}17]{local_beh}
S.~{Kleine}.
\newblock {Local behavior of Iwasawa's invariants.}
\newblock {\em {Int. J. Number Theory}}, 13(4):1013--1036, 2017.

\bibitem[Lam15]{lamplugh}
J.~Lamplugh.
\newblock An analogue of the {W}ashington-{S}innott theorem for elliptic curves
  with complex multiplication {I}.
\newblock {\em J. Lond. Math. Soc. (2)}, 91(3):609--642, 2015.

\bibitem[{Lan}97]{lang_diophantine}
S.~{Lang}.
\newblock {\em {Survey of diophantine geometry. Transl. from the Russian. Corr.
  2nd printing}}.
\newblock Berlin: Springer, corr. 2nd printing edition, 1997.

\bibitem[Lim17a]{lim}
M.~F. Lim.
\newblock Comparing the {$\pi$}-primary submodules of the dual {S}elmer groups.
\newblock {\em Asian J. Math.}, 21(6):1153--1181, 2017.

\bibitem[Lim17b]{lim_fuer_howson}
M.~F. Lim.
\newblock Notes on the fine {S}elmer groups.
\newblock {\em Asian J. Math.}, 21(2):337--361, 2017.

\bibitem[LKM16]{lim-murty}
M.~F. Lim and V.~Kumar~Murty.
\newblock The growth of fine {S}elmer groups.
\newblock {\em J. Ramanujan Math. Soc.}, 31(1):79--94, 2016.

\bibitem[LS18]{lim_congruent}
M.~F. Lim and R.~Sujatha.
\newblock Fine {S}elmer groups of congruent {G}alois representations.
\newblock {\em J. Number Theory}, 187:66--91, 2018.

\bibitem[Mat55]{mattuck}
A.~Mattuck.
\newblock Abelian varieties over {$p$}-adic ground fields.
\newblock {\em Ann. of Math. (2)}, 62:92--119, 1955.

\bibitem[MR04]{kolyvagin}
B.~Mazur and K.~Rubin.
\newblock Kolyvagin systems.
\newblock {\em Mem. Amer. Math. Soc.}, 168(799):viii+96, 2004.

\bibitem[NSW08]{nsw}
J.~{Neukirch}, A.~{Schmidt}, and K.~{Wingberg}.
\newblock {\em {Cohomology of number fields. 2nd ed.}}
\newblock Berlin: Springer, 2nd edition, 2008.

\bibitem[Pon20]{ponsinet}
G.~Ponsinet.
\newblock On the structure of signed {S}elmer groups.
\newblock {\em Math. Z.}, 294(3-4):1635--1658, 2020.

\bibitem[Sha09]{sharma}
A.~C. Sharma.
\newblock Iwasawa invariants for the false-{T}ate extension and congruences
  between modular forms.
\newblock {\em J. Number Theory}, 129(8):1893--1911, 2009.

\bibitem[SW18]{ramdorai-witte}
R.~Sujatha and M.~Witte.
\newblock Fine {S}elmer groups and isogeny invariance.
\newblock In {\em Geometry, algebra, number theory, and their information
  technology applications}, volume 251 of {\em Springer Proc. Math. Stat.},
  pages 419--444. Springer, Cham, 2018.

\bibitem[Ven02]{venjakob}
O.~Venjakob.
\newblock {On the Structure Theory of the Iwasawa Algebra of a $p$-adic Lie
  Group}.
\newblock {\em J. Eur. Math. Soc.}, (4):271--311, 2002.

\end{thebibliography}
\bibliographystyle{alpha}

\end{document}